\documentclass[10pt]{elsarticle}
\usepackage[utf8]{inputenc}
\usepackage[english]{babel}
\usepackage{amsmath}
\usepackage{amssymb}
\usepackage{amsfonts}
\usepackage{mathtools}
\usepackage{amsthm}
\usepackage{dsfont}
\usepackage{rotating}
\usepackage{graphicx}
\usepackage{floatflt,epsfig}
\usepackage{lineno,hyperref}
\usepackage{enumerate}
\usepackage{colortbl}
\usepackage{array,tabularx,tabulary,booktabs}
\usepackage{longtable}
\usepackage{multirow}
\usepackage{wrapfig}
\usepackage{subcaption}
\usepackage[table]{xcolor}
\newcolumntype{^}{>{\currentrowstyle}}

\journal{Arxiv}
\newtheorem{lemma}{Lemma}

\newtheorem{theorem}{Theorem}

\newtheorem{problem}{Problem}
\newtheorem{example}{Example}
\begin{document}
\renewcommand{\abstractname}{Abstract}
\renewcommand{\refname}{References}
\renewcommand{\tablename}{Figure.}
\renewcommand{\arraystretch}{0.9}
\thispagestyle{empty}
\sloppy

\begin{frontmatter}
\title{Minimum supports of eigenfunctions with the second largest eigenvalue of the Star graph\tnoteref{grant}}
\tnotetext[grant]{The reported study was funded by RFBR according to the research project 17-51-560008. The second and the fourth authors are partially supported by the program of fundamental scientific research of the SB RAS N I.5.1, project No.0314-2019-0016.}

\author[01]{Vladislav Kabanov}
\ead{vvk@imm.uran.ru}

\author[03,04]{Elena~V.~Konstantinova}
\ead{e\_konsta@math.nsc.ru}

\author[01,02]{Leonid~Shalaginov}
\ead{44sh@mail.ru}

\author[03,04]{Alexandr~Valyuzhenich\corref{cor1}}
\cortext[cor1]{Corresponding author}
\ead{graphkiper@mail.ru}

\address[01]{Krasovskii Institute of Mathematics and Mechanics, S. Kovalevskaja st. 16 \\ Yekaterinburg
620990, Russia}
\address[02] {Chelyabinsk State University, Brat'ev Kashirinyh st. 129\\Chelyabinsk  454021, Russia}
\address[03]{Sobolev Institute of Mathematics, Ak. Koptyug av. 4, Novosibirsk 630090, Russia}
\address[04]{Novosibirsk State University, Pirogova str. 2, Novosibirsk, 630090, Russia}


\begin{abstract}
The Star graph $S_n$, $n\ge 3$, is the Cayley graph on the symmetric group $Sym_n$
generated by the set of transpositions $\{(12),(13),\ldots,(1n)\}$.
In this work we study eigenfunctions of $S_n$ corresponding to the second largest eigenvalue $n-2$.
For $n\ge 8$ and $n=3$, we find the minimum cardinality of the support of an eigenfunction of $S_n$ corresponding to the second largest eigenvalue and obtain a characterization of eigenfunctions with the minimum cardinality
of the support.
\end{abstract}

\begin{keyword}
Star graph; eigenfunction; eigenspace; minimum support; the second largest eigenvalue; completely regular code
\vspace{\baselineskip}
\MSC[2010] 05C50\sep 05C25\sep 05E15\sep 05B30
\end{keyword}
\end{frontmatter}

\section{Introduction}
{\em The Star graph} $S_n=Cay(Sym_n,S)$, $n\ge 3$, is the Cayley graph on the symmetric group $Sym_n$  with the generating set $S =\{(1~i) ~|~ i \in \{2,\ldots,n\}\}$. It is a connected bipartite $(n-1)$-regular graph of order $n!\,$ and diameter $diam(S_n)=\lfloor \frac{3(n-1)}{2}\rfloor$ ~\cite{AK89}. Since this graph is bipartite it does not contain odd cycles but it does contain all even $l$-cycles where $l=6,8,\ldots,n!$  (with the sole exception when $l=4$)~\cite{JLD91} which means that $S_n$ is hamiltonian.

The spectrum of the Star graph is integral \cite{ChapuyFeray,KrakovskiMohar}. More precisely, for $n \ge 3$ and for each integer $1\le k \le n-1$, the values $\pm (n-k)$ are eigenvalues of $S_n$; if $n\ge 4$, then $0$ is an eigenvalue of $S_n$. Since the Star graph is bipartite,
$\rm{mul}(n-k)=\rm{mul}(-n+k)$ for each integer $1\le k\le n$. Moreover, $\pm (n-1)$ are simple eigenvalues of $S_n$. A lower bound on multiplicities of the eigenvalues was found as $\left({n-2}\atop{k-1}\right)$~\cite{KrakovskiMohar}, and it was improved as~$\left({n-2}\atop{n-k-1}\right)\left({n-1}\atop{n-k}\right)$~\cite{ChapuyFeray} for any $1\le k \le n-1$. Later a method for getting explicit formulas for multiplicities of eigenvalues $\pm(n-k)$ in the Star graphs $S_n$ was suggested~\cite{AKK16,KK15}, and the behavior of the eigenvalues multiplicity function of the Star graph $S_n$ for eigenvalues $\pm(n-k)$ where $1\leq k \leq \frac{n+1}{2}$ was investigated~\cite{K18}. It was shown that the function has a polynomial behavior on $n$. Computational results showed that the same polynomial behavior of the eigenvalues multiplicity function occurs for any integers $n \geq 2$ and $1\leq k \leq n$. Moreover, explicit formulas for calculating multiplicities of eigenvalues $\pm(n-k)$ where $2\leq k \leq 12$ were found. In particular, $\rm{mul}(n-2)=\rm{mul}(2-n)=(n-1)(n-2)$.

We investigate the following problem.
\begin{problem}\label{MinSupport}
For a graph $\Gamma$ and its eigenvalue $\lambda$ to find the minimum cardinality of the support of a $\lambda$-eigenfunction of $\Gamma$.
\end{problem}
In many cases Problem~\ref{MinSupport} is directly related to the problem of finding the minimum possible difference of two combinatorial objects and to the problem of finding the minimum cardinality of the bitrades. In more details, these connections are described in~\cite{Krotovtezic,KrotovMogilnykhPotapov}. Problem~\ref{MinSupport} was studied for the bilinear forms of graphs in~\cite{SotnikovaBilinear}, for the cubical distance-regular graphs in~\cite{SotnikovaCubical}, for the Doob graphs in~\cite{Bespalov}, for the Grassmann graphs in \cite{KrotovMogilnykhPotapov}, for the Hamming graphs in~\cite{Krotovtezic,Potapov,Valyuzhenich,ValyuzhenichVorobev,VorobevKrotov}, for the Johnson graphs in~\cite{VMV} and for the Paley graphs in~\cite{GoryainovKabanovShalaginovValyuzhenich}.

We consider Problem \ref{MinSupport} for the Star graph $S_n$ and its eigenvalue $n-2$.
We find the minimum cardinality of the support of $(n-2)$-eigenfunctions of $S_n$ and give a characterization of $(n-2)$-eigenfunctions with the minimum cardinality of the support for $n\ge 8$ and $n=3$.
We also show that for $n\ge 8$ and $n=3$ an arbitrary $(n-2)$-eigenfunction of $S_n$
with the minimum cardinality of the support is the difference of the characteristic functions of two completely regular codes of covering radius $2$.

The paper is organized as follows. In Section \ref{Preliminaries}, we introduce basic definitions and give some preliminary results.
In Section \ref{SectionM(f)}, we reduce Problem \ref{MinSupport} for the Star graph $S_n$ and its eigenvalue $n-2$ to some extremal
problem on the set of real $n\times n$ matrices.
In Section \ref{SectionExtremalProblem}, we solve this extremal problem.
In Section \ref{SectionMainTheorem}, for $n\ge 8$ and $n=3$ we prove that the minimum cardinality of the support of an $(n-2)$-eigenfunction of $S_n$ is $2(n-1)!$ and give a characterization of eigenfunctions with the minimum cardinality
of the support.
In Section \ref{SectionCompletelyRegularCodes}, for $n\ge 8$ and $n=3$ we show that an arbitrary $(n-2)$-eigenfunction of $S_n$ with the minimum cardinality of the support is the difference of the characteristic functions of two completely regular codes of covering radius $2$.
\section{Preliminaries}\label{Preliminaries}
\subsection{Star graph}

Let $G$ be a group and $S$ be an inverse-closed identity-free generating set in $G$. The \emph{Cayley graph on} $G$ \emph{with the generating set} $S$ (denoted by $Cay(G,S)$) is the graph whose vertices are the elements of $G$, and any two elements $x,y \in G$ are adjacent in $Cay(G,S)$ whenever $xy^{-1} \in S$. For a positive integer $n \ge 3$, the \emph{Star graph} (denoted by $S_n$) is the Cayley graph on the symmetric group $Sym_n$ with the generating set $S =\{(1~i) ~|~ i \in \{2,\ldots,n\}\}$.

\subsection{Eigenfunctions of graphs}

Let $\Gamma=(V,E)$ be a graph with the adjacency matrix $A(\Gamma)$. The set of neighbors of a vertex $x$ is denoted by $N(x)$. Let $\lambda$ be an eigenvalue of the matrix $A(\Gamma)$. A function $f:V\longrightarrow{\mathbb{R}}$ is called a {\em $\lambda$-eigenfunction} of $\Gamma$ if $f\not\equiv 0$ and the equality
$$\lambda\cdot f(x)=\sum_{y\in{N(x)}}f(y)$$ holds for any $x\in V$. Note that the vector of values of a $\lambda$-eigenfunction is an eigenvector of $A(\Gamma)$ with eigenvalue $\lambda$.
The {\em support} of a function $f:V\longrightarrow{\mathbb{R}}$ is the set $Supp(f)=\{x\in V~|~f(x)\neq 0\}$.
For a function  $f:V\longrightarrow{\mathbb{R}}$ denote $E(f)=\{y\in \mathbb{R}~|~y=f(x),x\in V\}$.

Let $u\in\{1,\ldots,n\}$ and $v,w\in\{2,\ldots,n\}$, where $v\neq w$.
We define the function $f_{u}^{v,w}:Sym_{n}\longrightarrow{\mathbb{R}}$ by the following rule:
$$
f_{u}^{v,w}(\pi)=\begin{cases}
1,&\text{if $\pi(v)=u$;}\\
-1,&\text{if $\pi(w)=u$;}\\
0,&\text{otherwise.}
\end{cases}
$$

The following result is a particular case of Proposition 1 proved in \cite{GoryainovKabanovKonstantinovaShalaginovValyuzhenich}.

\begin{lemma}\label{f_{i}^{j,k}function}
Let $u\in\{1,\ldots,n\}$ and $v,w\in\{2,\ldots,n\}$, where $v\neq w$. Then for $n\geq 3$, the function $f_{u}^{v,w}$ is an $(n-2)$-eigenfunction of $S_n$.
\end{lemma}

Denote $$\mathcal{F}_2=\{f_{u}^{2,w}~|~u\in\{2,\ldots,n\},w\in\{3,\ldots,n\}\}.$$

The following lemma was proved in \cite{GoryainovKabanovKonstantinovaShalaginovValyuzhenich}.

\begin{lemma}[\cite{GoryainovKabanovKonstantinovaShalaginovValyuzhenich}, Lemma 15]\label{basis}
For $n\ge 3$, the set $\mathcal{F}_2$ forms a basis of the eigenspace of $S_n$ with eigenvalue $n-2$.
\end{lemma}

Denote $$\mathcal{F}=\{f_{u}^{v,w}~|~u\in\{1,\ldots,n\},v,w\in\{2,\ldots,n\},v\neq w\}.$$
In Section \ref{SectionMainTheorem} we prove that if $f$ is an $(n-2)$-eigenfunction of $S_n$, then $|Supp(f)|\ge 2(n-1)!$. Moreover, we prove that $|Supp(f)|=2(n-1)!$ if and only if $f=c\cdot \tilde{f}$, where $c$ is a constant and $\tilde{f}\in \mathcal{F}$.

\subsection{Matrices}

Let  $M=(m_{i,j})$ be a real $n\times n$ matrix.
We say that $M$ is {\em special} if $M$ is non-zero and the following conditions hold:
\begin{enumerate}
  \item $m_{i,1}=0$ for any $i\in{\{1,\ldots,n\}}$.
  \item $m_{1,j}=0$ for any $j\in{\{1,\ldots,n\}}$.
  \item $\sum_{j=1}^{n}m_{i,j}=0$ for any $i\in{\{1,\ldots,n\}}$.
\end{enumerate}

\begin{example}\label{}
The matrix $$M=
\begin{pmatrix}
0 & 0 & 0 & 0\\
0 & 1 & 0 & -1\\
0 & 3 & -1 & -2\\
0 & -1 & -1 & 2
\end{pmatrix}$$
is special.
\end{example}

Let $M=(m_{i,j})$ be a real $n\times n$ matrix and let $X$ be a subset of $Sym_n$.
Denote $$g_M(n)=|\{\pi\in Sym_n~|~\sum_{i=1}^{n}m_{i,\pi(i)}\neq 0\}|$$ and
$$g_M(X)=|\{\pi\in X~|~\sum_{i=1}^{n}m_{i,\pi(i)}\neq 0\}|.$$

For an $n\times n$ matrix $M$ and $\alpha\in\{1,\ldots,n\}$ denote by $R_{\alpha}$ the $\alpha$-th row of $M$.

\subsection{Equitable partitions and completely regular codes}
Let $\Gamma=(V,E)$ be a graph.
An ordered partition $(C_1,\ldots, C_{r})$ of $V$ is called {\em equitable} if for any
$i,j\in \{1,\ldots,r\}$ there is $S_{i,j}$ such that any vertex of $C_i$
has exactly $S_{i,j}$ neighbors in $C_j$.
The matrix $S=(S_{i,j})_{i,j\in
\{1,\ldots,r\}}$ is called the {\em quotient matrix} of the
equitable partition.
A set $C\subseteq V$ is called a {\em completely regular code} in $\Gamma$ if the partition $(C^{(0)},\ldots,C^{(\rho)})$
is equitable, where $C^{(d)}$ is the set of vertices at distance $d$ from $C$ and $\rho$ ({\em the covering radius} of $C$) is the
maximum $d$ for which $C^{(d)}$ is nonempty.
In other words, a subset of $V$ is a completely regular code in $\Gamma$
if the distance partition with respect to the subset is equitable.
For more information on equitable partitions and completely regular codes see \cite{BorgesRifaZinoviev,Godsil}.

\section{Correspondence between the values of an $(n-2)$-eigenfunction of $S_n$ and the diagonals of an $n\times n$ matrix}\label{SectionM(f)}

In this section, for an arbitrary $(n-2)$-eigenfunction $f$ of $S_n$, we introduce a special $n\times n$ matrix $M(f)$ and match the permutations from $Sym_n$ with diagonals of $M(f)$ in such a way that the value of $f$ on a permutation $\pi$ is the sum of elements of the corresponding diagonal of $M(f)$.

Let $f$ be an $(n-2)$-eigenfunction of $S_n$. By Lemma \ref{basis}, there exist the numbers $\mu_{i}^{j}(f)\in \mathbb{R}$, where $i\in{\{2,\ldots,n\}}$ and $j\in{\{3,\ldots,n\}}$, such that $$f=\sum_{\substack{j\in{\{3,\ldots,n\}}\\ i\in{\{2,\ldots,n\}}}}\mu_{i}^{j}(f)\cdot f_{i}^{2,j}.$$
We define the matrix $M(f)=(m_{i,j}(f))_{i,j\in{\{1,\ldots,n\}}}$ by the following rule:
\begin{equation}\label{M(f)Definition}
m_{i,j}(f)=\begin{cases}
-\mu_{i}^{j}(f),&\text{if $i>1$ and $j>2$;}\\
\sum_{s=3}^{n}\mu_{i}^{s}(f),&\text{if $i>1$ and $j=2$;}\\
0,&\text{if $i=1$ or $j=1$.}
\end{cases}
\end{equation}

\begin{lemma}\label{f(pi)}
Let $f$ be an $(n-2)$-eigenfunction of $S_n$. Then
$$f(\pi)=\sum_{i=1}^{n}m_{i,\pi^{-1}(i)}(f)$$ for any $\pi\in Sym_n$.
\end{lemma}
\begin{proof}
For $i\in{\{2,\ldots,n\}}$ denote $$f_i=\sum_{j=3}^{n}\mu_{i}^{j}(f)\cdot f_{i}^{2,j}.$$
By the definition of $f_{u}^{v,w}$ we have
\begin{equation}\label{f_i(pi)}
f_i(\pi)=\begin{cases}
-\mu_{i}^{j}(f),&\text{if $\pi(j)=i$ and $j\in{\{3,\ldots,n\}}$;}\\
\sum_{s=3}^{n}\mu_{i}^{s}(f),&\text{if $\pi(2)=i$;}\\
0,&\text{if $\pi(1)=i$.}
\end{cases}
\end{equation}

Using the equalities (\ref{M(f)Definition}) and (\ref{f_i(pi)}), we see that $f_i(\pi)=m_{i,\pi^{-1}(i)}(f)$ for any $i\in{\{2,\ldots,n\}}$ and $\pi\in Sym_n$.
Then $$f(\pi)=\sum_{i=2}^{n}f_i(\pi)=\sum_{i=2}^{n}m_{i,\pi^{-1}(i)}(f)$$ for any $\pi\in Sym_n$.
Since $m_{1,\pi^{-1}(1)}(f)=0$ due to the definition of $M(f)$, we obtain that $$f(\pi)=\sum_{i=1}^{n}m_{i,\pi^{-1}(i)}(f).$$
\end{proof}

\begin{lemma}\label{E(f)={}}
Let $f$ be an $(n-2)$-eigenfunction of $S_n$. Then $$E(f)=\left\{\sum_{i=1}^{n}m_{i,\pi(i)}(f)~|~\pi\in Sym_n\right\}.$$
\end{lemma}
\begin{proof}
We note that $$\left\{\sum_{i=1}^{n}m_{i,\pi(i)}(f)~|~\pi\in Sym_n\right\}=\left\{\sum_{i=1}^{n}m_{i,\pi^{-1}(i)}(f)~|~\pi\in Sym_n\right\}.$$
Then by Lemma \ref{f(pi)} we have $$E(f)=\left\{\sum_{i=1}^{n}m_{i,\pi(i)}(f)~|~\pi\in Sym_n\right\}.$$
\end{proof}

Using Lemma \ref{E(f)={}} and the definition of $g_M(n)$, we immediately obtain the following result.

\begin{lemma}\label{Supp(f)=g_M(n)}
Let $f$ be an $(n-2)$-eigenfunction of $S_n$. Then $|Supp(f)|=g_{M(f)}(n)$.
\end{lemma}

\section{Extremal problem on the set of all special $n\times n$ matrices}\label{SectionExtremalProblem}

In view of Lemma \ref{Supp(f)=g_M(n)}, in order to solve Problem \ref{MinSupport} for the Star graph $S_n$ and its eigenvalue $n-2$, it suffices to find the minimum value of $g_{M(f)}(n)$, where $n$ is fixed and $f$ is an arbitary $(n-2)$-eigenfunction of $S_n$. Since the matrix $M(f)$ is special, in this section we focus on the following extremal problem formulated for the class of special matrices.

\begin{problem}\label{MinSpecMatr}
Given a positive integer $n$, to find the minimum value of $g_M(n)$ for the class of special $n \times n$ matrices $M$.
\end{problem}

In this section, we solve Problem \ref{MinSpecMatr} (see Theorem 1) and prove that $g_M(n)\ge 2(n-1)!$ holds for any special $n \times n$ matrix $M$ with $n \ge 8$ or $n = 3$. We then prove that this bound is tight and classify the special matrices in the equality case. This finally leads to a solution of Problem \ref{MinSupport} for the Star graph $S_n$ and its eigenvalue $n-2$ (see Theorem \ref{MainTheorem}).

Let $M=(m_{i,j})$ be a real $n\times n$ matrix and let $(A_1,\ldots,A_t)$ be a partition of $\{1,\ldots,n\}$, where $t\ge 2$.
Let $\alpha,\beta\in \{1,\ldots,n\}$ and $\alpha\neq \beta$.
We say that $R_{\alpha}$ and $R_{\beta}$ have the {\em $(A_1,\ldots,A_t)$-property} if for any $k,m\in{\{1,\ldots,t\}}$, $k\neq m$ and for any
$a\in{A_k}$, $b\in{A_m}$ the condition
$$m_{\alpha,a}+m_{\beta,b}\neq m_{\alpha,b}+m_{\beta,a}$$
holds.

\begin{example}\label{2}
Let us consider the matrix $$M=
\begin{pmatrix}
0 & 0 & 0 & 0\\
0 & 1 & -1 & 0\\
2 & 2 & 0 & 0\\
0 & 0 & 3 & 3
\end{pmatrix}.$$
\begin{enumerate}
    \item The rows $R_1$ and $R_2$ have $(A_1,A_2,A_3)$-property, where $A_1=\{1,4\}$, $A_2=\{2\}$ and $A_3=\{3\}$.
    \item The rows $R_1$ and $R_3$ have $(A_1,A_2)$-property, where $A_1=\{1,2\}$ and $A_2=\{3,4\}$.
    \item The rows $R_3$ and $R_4$ have $(A_1,A_2)$-property, where $A_1=\{1,2\}$ and $A_2=\{3,4\}$.
\end{enumerate}
\end{example}

In Lemmas \ref{Lemma1} and \ref{Lemma2}, we obtain lower bounds for $g_M(n)$, where $M$ is an arbitrary real matrix having a pair of rows with the $(A_1,\ldots,A_t)$-property.

Let  $\{a_1,\ldots,a_{\ell}\}$ and $\{\alpha_1,\ldots,\alpha_{\ell}\}$ be subsets of $\{1,\ldots,n\}$ for some $1\le \ell\le n$.
Denote by $S_{a_1,\ldots,a_{\ell}}^{\alpha_1,\ldots,\alpha_{\ell}}$ the set of permutations $\pi\in Sym_n$ such that
$\pi(\alpha_i)=a_i$ for any $i\in\{1,\ldots,\ell\}$.

\begin{lemma}\label{Lemma1}
Let $M=(m_{i,j})$ be a real $n\times n$ matrix. Let $\{c_1,\ldots,c_h,a,b\}$ and $\{\gamma_1,\ldots,\gamma_{h},\alpha,\beta\}$ be subsets of $\{1,\ldots,n\}$ for some $0\le h\le n-2$.
Suppose $R_{\alpha}$ and $R_{\beta}$ have the $(A_1,\ldots,A_t)$-property and $a\in A_k$, $b\in A_m$ for some $k,m\in{\{1,\ldots,t\}}$ and $k\neq m$. Then
$$g_M\left(S_{c_1,\ldots,c_h,a,b}^{\gamma_1,\ldots,\gamma_{h},\alpha,\beta}\cup S_{c_1,\ldots,c_h,b,a}^{\gamma_1,\ldots,\gamma_{h},\alpha,\beta}\right)\ge (n-h-2)!.$$
\end{lemma}
\begin{proof}
Denote $Y_1=S_{c_1,\ldots,c_h,a,b}^{\gamma_1,\ldots,\gamma_{h},\alpha,\beta}$ and $Y_2=S_{c_1,\ldots,c_h,b,a}^{\gamma_1,\ldots,\gamma_{h},\alpha,\beta}$.
We note that for any permutation $\pi\in Y_1$ there is a unique permutation  $\pi'\in Y_2$ such that $\pi(s)=\pi'(s)$ for any $s\in{\{1,2,\ldots,n\}}\setminus \{\alpha,\beta\}$.
Then
$$\sum_{i=1}^{n}m_{i,\pi(i)}-\sum_{i=1}^{n}m_{i,\pi'(i)}=m_{\alpha,a}+m_{\beta,b}-m_{\alpha,b}-m_{\beta,a}$$ for any $\pi\in Y_1$.
Since $R_{\alpha}$ and $R_{\beta}$ have the $(A_1,\ldots,A_t)$-property, $a\in A_k$ and $b\in A_m$, we have
$$\sum_{i=1}^{n}m_{i,\pi(i)}\neq \sum_{i=1}^{n}m_{i,\pi'(i)}.$$
Therefore, $\sum_{i=1}^{n}m_{i,\pi(i)}\neq 0$ or $\sum_{i=1}^{n}m_{i,\pi'(i)}\neq 0$ for any $\pi\in Y_1$.
Using the equality $$|Y_1|=|Y_2|=(n-h-2)!,$$ we obtain
$$g_M(Y_1\cup Y_2)\ge (n-h-2)!.$$
\end{proof}

\begin{lemma}\label{Lemma2}
Let $M$ be a real $n\times n$ matrix, $\alpha,\beta\in \{1,\ldots,n\}$ and $\alpha\neq \beta$. Suppose $R_{\alpha}$ and $R_{\beta}$ have the $(A_1,\ldots,A_t)$-property. Then
$$g_M(n)\geq \left(\sum_{1\leq k<m\leq t}|A_k|\cdot|A_m|\right)\cdot(n-2)!.$$
\end{lemma}
\begin{proof}
Let $k,m\in{\{1,\ldots,t\}}$ and $k<m$. Let us consider arbitrary $a\in{A_k}$ and $b\in{A_m}$. Lemma \ref{Lemma1} implies that
$g_M(S_{a,b}^{\alpha,\beta}\cup S_{b,a}^{\alpha,\beta})\geq (n-2)!$.
Denote $$X=\bigcup_{\substack{1\leq k<m\leq t\\ a\in{A_k},b\in{A_m}}}\left(S_{a,b}^{\alpha,\beta}\cup S_{b,a}^{\alpha,\beta}\right).$$
Then we have $$g_M(X)=\sum_{\substack{1\leq k<m\leq t\\ a\in{A_k},b\in{A_m}}}g_M(S_{a,b}^{\alpha,\beta}\cup S_{b,a}^{\alpha,\beta})\geq
\left(\sum_{1\leq k<m\leq t}|A_k|\cdot|A_m|\right)\cdot(n-2)!.$$
Thus, we obtain $$g_M(n)\geq g_M(X)\geq \left(\sum_{1\leq k<m\leq t}|A_k|\cdot|A_m|\right)\cdot(n-2)!.$$
\end{proof}

\begin{lemma}\label{Arifmetica}
Let $n=n_1+\ldots+n_t$, where $n_i\in{\mathbb{N}}$ for any $i\in{\{1,\ldots,t\}}$, $n_1\geq\ldots\geq n_t$, $n\geq 7$ and $t\geq 3$.
Then either $$\sum_{1\leq k<m\leq t}n_{k}n_{m}>2(n-1)$$ or
\begin{center}
$t=3$, $n_1=n-2$ and $n_2=n_3=1$.
\end{center}
\end{lemma}
\begin{proof}
 Denote $$S=\sum_{1\leq k<m\leq t}n_{k}n_{m}.$$ We consider three cases.

 In the first case we suppose that $3\leq n_1 \leq n-3$.
 Then we have $$S\geq n_{1}(n_2+\ldots+n_t)+n_{2}n_{3}\ge 3(n-3)+1>2(n-1).$$

 In the second case we suppose that $n_1\geq n-2$. Since $t\geq 3$, it is possible only if $t=3$, $n_1=n-2$ and $n_2=n_3=1$.

 In the third case we suppose that $n_1\leq 2$.
 Let $i=|\{1\le k\le t~|~n_k=2\}|$.

 If $i\geq 2$, then $n_1=n_2=2$ and $$S\geq n_1(n_2+\ldots+n_t)+n_2(n_3+\ldots+n_t)=4n-12>2(n-1).$$
 If $n-i\geq 3$, then $n_t=n_{t-1}=n_{t-2}=1$ and $$S\geq n_{t}(n_1+\ldots+n_{t-1})+n_{t-1}(n_1+\ldots+n_{t-2})+n_{t-2}(n_1+\ldots+n_{t-3})=3n-6>2(n-1).$$
Since $n\geq 7$, $i\geq 2$ or $n-i\geq 3$ and we obtain the case considered above.
\end{proof}

Let $M=(m_{i,j})$ be a real $n\times n$ matrix and $\alpha\in \{1,\ldots,n\}$.
Let $x,y\in \mathbb{R}$ and $r_1,r_2,s\in \{2,\ldots,n\}$, where $x,y\neq 0$ and $r_1\neq r_2$.
We say that $R_{\alpha}$ is the {\em $(x,r_1,r_2)$-row} if $m_{\alpha,r_1}=x$, $m_{\alpha,r_2}=-x$ and $m_{\alpha,j}=0$ for any $j\in \{1,\ldots,n\}\setminus \{r_1,r_2\}$.
We say that $R_{\alpha}$ is the {\em $(y,s)$-row} if $m_{\alpha,s}=(n-2)y$, $m_{\alpha,1}=0$ and $m_{\alpha,j}=-y$ for any $j\in \{1,\ldots,n\}\setminus \{1,s\}$.

\begin{example}\label{3}
Let $$M=
\begin{pmatrix}
0 & 0 & 0 & 0\\
0 & 1 & 0 & -1\\
0 & 0 & 0 & 0\\
0 & -2 & 4 & -2
\end{pmatrix}.$$
Then $R_2$ is the $(1,2,4)$-row and $R_4$ is the $(2,3)$-row.
\end{example}

\begin{lemma}\label{stringtype}
Let $M$ be a special $n\times n$ matrix, $g_M(n)\leq 2(n-1)!$, $n\geq 7$ and $\alpha\in \{1,\ldots,n\}$. Suppose $R_{\alpha}$ is a non-zero row. Then $R_{\alpha}$ is the $(x,r_1,r_2)$-row or the $(y,s)$-row.
\end{lemma}
\begin{proof}
Suppose that $R_{\alpha}$ consists of the distinct elements $z_1,\ldots,z_t$ each of them $z_k$, where $1\le k \le t$, occurs $n_k$ times in $R_{\alpha}$. Without loss of generality, we assume that $n_1\geq\ldots\geq n_t$.
If $t\leq 2$, then by the definition of special matrix we obtain that all elements of $R_{\alpha}$ are zeroes.
So, we can assume that $t\geq 3$.

For $k\in{\{1,\ldots,t\}}$ denote $$A_k=\{j\in{\{1,\ldots,n\}}~|~m_{\alpha,j}=z_k\}.$$
We note that $|A_k|=n_k$ for any $k\in{\{1,\ldots,t\}}$.
Since $M$ is special, all elements of $R_1$ are zeroes. Hence $R_1$ and $R_{\alpha}$ have $(A_1,\ldots,A_t)$-property.
Lemma \ref{Lemma2} implies that $$g_M(n)\geq \left(\sum_{1\leq k<m\leq t}n_{k}n_{m}\right)\cdot(n-2)!.$$
On the other hand, we have $g_M(n)\leq 2(n-1)!$. So $$\sum_{1\leq k<m\leq t}n_{k}n_{m}\leq 2(n-1).$$
Recall that in the beginning of the proof we
assumed $t\geq 3$. Then by Lemma \ref{Arifmetica} we obtain that $t=3$, $n_1=n-2$ and $n_2=n_3=1$.
Therefore, by the definition of special matrix we obtain that $R_{\alpha}$ is the $(x,r_1,r_2)$-row or the $(y,s)$-row.
\end{proof}

\begin{lemma}\label{(A,B)Property}
 Let $M$ be a special $n\times n$ matrix and $g_M(n)\leq 2(n-1)!$, where $n\ge 8$. Suppose $\alpha,\beta\in \{1,\ldots,n\}$, $\alpha\neq \beta$ and $R_{\alpha}\neq R_{\beta}$.
Then there exists a partition $(A,B)$ of $\{1,\ldots,n\}$ such that $R_{\alpha}$ and $R_{\beta}$ have the $(A,B)$-property and $|A|=2$.
\end{lemma}
\begin{proof}
Firstly, we prove that $R_{\alpha}$ and $R_{\beta}$ have the $(A,B)$-property, where $|A|\in \{2,3,4\}$.
For a set $X\subseteq \{1,\ldots,n\}$ denote $\overline{X}=\{1,\ldots,n\}\setminus X$.
By Lemma \ref{stringtype} we have four cases for $R_{\alpha}$ and $R_{\beta}$.

\textbf{Case 1.} Suppose $R_{\alpha}$ or $R_{\beta}$ is a zero row. Without loss of generality, we assume that $R_{\alpha}$ is a zero row.
Let us consider two subcases.

\textbf{1.1.} $R_{\beta}$ is the $(x,r_1,r_2)$-row. Let $A=\{r_1,r_2\}$ and $B=\overline{A}$. Then $R_{\alpha}$ and $R_{\beta}$ have the $(A,B)$-property and $|A|=2$.

\textbf{1.2.} $R_{\beta}$ is the $(y,s)$-row. Let $A=\{1,s\}$ and $B=\overline{A}$. Then $R_{\alpha}$ and $R_{\beta}$ have the $(A,B)$-property and $|A|=2$.

\textbf{Case 2.} Suppose $R_{\alpha}$ is the $(x_1,r_1,r_2)$-row and $R_{\beta}$ is the $(x_2,r_3,r_4)$-row.
Let us consider four subcases depending on $|\{r_1,r_2\}\cap \{r_3,r_4\}|$.

\textbf{2.1.} $\{r_1,r_2\}\cap \{r_3,r_4\}=\emptyset$. Let $A=\{r_1,r_2,r_3,r_4\}$ and $B=\overline{A}$. Then $R_{\alpha}$ and $R_{\beta}$ have the $(A,B)$-property and $|A|=4$.

\textbf{2.2.} $|\{r_1,r_2\}\cap \{r_3,r_4\}|=1$ and $x_1\neq x_2$. Without loss of generality, we assume that $r_1=r_3$. Let $A=\{r_1,r_2,r_4\}$ and $B=\overline{A}$.
Then $R_{\alpha}$ and $R_{\beta}$ have the $(A,B)$-property and $|A|=3$.

\textbf{2.3.} $|\{r_1,r_2\}\cap \{r_3,r_4\}|=1$ and $x_1=x_2$. Without loss of generality, we assume that $r_1=r_3$. Let $A=\{r_2,r_4\}$ and $B=\overline{A}$.
Then $R_{\alpha}$ and $R_{\beta}$ have the $(A,B)$-property and $|A|=2$.

\textbf{2.4.} $|\{r_1,r_2\}\cap \{r_3,r_4\}|=2$. Without loss of generality, we assume that $r_1=r_3$ and $r_2=r_4$. Since $R_{\alpha}\neq R_{\beta}$, we have $x_1\neq x_2$. Let $A=\{r_1,r_2\}$ and $B=\overline{A}$.
Then $R_{\alpha}$ and $R_{\beta}$ have the $(A,B)$-property and $|A|=2$.

\textbf{Case 3.} Suppose $R_{\alpha}$ is the $(y_1,s_1)$-row and $R_{\beta}$ is the $(y_2,s_2)$-row.
Let us consider three subcases.

\textbf{3.1.} $s_1\neq s_2$ and $y_1\neq y_2$. Let $A=\{1,s_1,s_2\}$ and $B=\overline{A}$. Then $R_{\alpha}$ and $R_{\beta}$ have the $(A,B)$-property and $|A|=3$.

\textbf{3.2.} $s_1\neq s_2$ and $y_1=y_2$. Let $A=\{s_1,s_2\}$ and $B=\overline{A}$.
Then $R_{\alpha}$ and $R_{\beta}$ have the $(A,B)$-property and $|A|=2$.

\textbf{3.3.} $s_1=s_2$. Since $R_{\alpha}\neq R_{\beta}$, we have $y_1\neq y_2$. Let $A=\{1,s_1\}$ and $B=\overline{A}$.
Then $R_{\alpha}$ and $R_{\beta}$ have the $(A,B)$-property and $|A|=2$.

\textbf{Case 4.} Suppose $R_{\alpha}$ is the $(x,r_1,r_2)$-row and $R_{\beta}$ is the $(y,s)$-row (the case when $R_{\alpha}$ is  the $(y,s)$-row  and $R_{\beta}$ is the $(x,r_1,r_2)$-row is similar).
Let us consider three subcases.

\textbf{4.1.} $s\not\in \{r_1,r_2\}$. Let $A=\{1,r_1,r_2,s\}$ and $B=\overline{A}$. Then $R_{\alpha}$ and $R_{\beta}$ have the $(A,B)$-property and $|A|=4$.

\textbf{4.2.} $s\in \{r_1,r_2\}$ and $x\neq (n-1)y$. Without loss of generality, we assume that $s=r_1$. Let $A=\{1,r_1,r_2\}$ and $B=\overline{A}$.
Then $R_{\alpha}$ and $R_{\beta}$ have the $(A,B)$-property and $|A|=3$.

\textbf{4.3.} $s\in \{r_1,r_2\}$ and $x=(n-1)y$. Without loss of generality, we assume that $s=r_1$. Let $A=\{1,r_2\}$ and $B=\overline{A}$.
Then $R_{\alpha}$ and $R_{\beta}$ have the $(A,B)$-property and $|A|=2$.

Thus, we prove that $R_{\alpha}$ and $R_{\beta}$ have the $(A,B)$-property, where $|A|\in \{2,3,4\}$.
Hence by Lemma \ref{Lemma2} we obtain that $g_M(n)\ge |A|\cdot|B|\cdot(n-2)!$.
On the other hand, we have $g_M(n)\le 2(n-1)!$. Therefore $|A|\cdot|B|\le 2(n-1)$.
So, $|A|=2$ and the cases when $|A|\in \{3,4\}$ do not hold. The lemma is proved.
\end{proof}

Let $M$ be an $n\times n$ matrix and $\theta\in \{1,\ldots,n\}$. We say that $M$ is {\em $\theta$-uniform} if $R_{\alpha}=R_{\beta}$ for any $\alpha,\beta\in \{1,\ldots,n\}\setminus \theta$.

\begin{lemma}\label{RavnomernostMatrix}
Let $M$ be a special $n\times n$ matrix and $g_M(n)\leq 2(n-1)!$, where $n\geq 8$ or $n=3$. Then $M$ is $\theta$-uniform for some $\theta\in{\{1,\ldots,n\}}$.
\end{lemma}
\begin{proof}
Firstly, let us consider the case $n=3$. Since $M$ is special, any non-zero row of $M$ is the $(x,2,3)$-row. Hence either $M$ is $\theta$-uniform for some $\theta\in{\{1,2,3\}}$ or $R_2$ and $R_3$ are $(x_1,2,3)$-row and $(x_2,2,3)$-row, where $x_1\neq x_2$. In the last case we have $g_M(n)=6$
and we obtain a contradiction with $g_M(n)\leq 2(n-1)!$.
In what follows, in this lemma we assume that $n\ge 8$.

Suppose that $M$ is not $\theta$-uniform for any $\theta\in{\{1,\ldots,n\}}$. Then there exist four rows $R_{\alpha}$, $R_{\beta}$, $R_{\gamma}$ and $R_{\delta}$ such that $R_{\alpha}\neq R_{\beta}$ and $R_{\gamma}\neq R_{\delta}$.

Since $R_{\alpha}\neq R_{\beta}$, by Lemma \ref{(A,B)Property} we obtain that $R_{\alpha}$ and $R_{\beta}$ have $(A,B)$-property, where $|A|=2$.
Let $a\in{A}$ and $b\in{B}$. Lemma \ref{Lemma1} implies that $g_M(S_{a,b}^{\alpha,\beta}\cup S_{b,a}^{\alpha,\beta})\geq (n-2)!$.
Denote $$X_1=\bigcup_{a\in{A},b\in{B}}(S_{a,b}^{\alpha,\beta}\cup S_{b,a}^{\alpha,\beta}).$$
Then we have

\begin{equation}\label{Neravenstvo1}
g_M(X_1)=\sum_{a\in{A},b\in{B}}g_M(S_{a,b}^{\alpha,\beta}\cup S_{b,a}^{\alpha,\beta})\geq |A|\cdot|B|\cdot (n-2)!=2(n-2)(n-2)!
\end{equation}

Since $R_{\gamma}\neq R_{\delta}$, by Lemma \ref{(A,B)Property} we obtain that $R_{\gamma}$ and $R_{\delta}$ have $(C,D)$-property, where $|C|=2$.
Let $c\in{C}$ and $d\in{D}$. Let us consider arbitrary $b',b''\in{B\setminus\{c,d\}}$, where $b'\neq b''$.
Lemma \ref{Lemma1} implies that $$g_M\left(S_{b',b'',c,d}^{\alpha,\beta,\gamma,\delta}\cup S_{b',b'',d,c}^{\alpha,\beta,\gamma,\delta}\right)\geq (n-4)!.$$
Denote $$X_2=\bigcup_{\substack{c\in{C},d\in{D}\\b',b''\in{B\setminus\{c,d\}}, b'\neq b''}}\left(S_{b',b'',c,d}^{\alpha,\beta,\gamma,\delta}\cup S_{b',b'',d,c}^{\alpha,\beta,\gamma,\delta}\right).$$
Then we have
\begin{equation}\label{Neravenstvo2}
\begin{split}
&g_M(X_2)=\sum_{\substack{c\in{C},d\in{D}\\b',b''\in{B\setminus\{c,d\}}, b'\neq b''}}g_M\left(S_{b',b'',c,d}^{\alpha,\beta,\gamma,\delta}\cup S_{b',b'',d,c}^{\alpha,\beta,\gamma,\delta}\right)\geq\\
&\geq|C|\cdot |D|\cdot (|B|-2)\cdot(|B|-3)\cdot (n-4)!=\\
&=2(n-2)(n-4)(n-5)(n-4)!.
\end{split}
\end{equation}

Note that $X_1\cap X_2=\emptyset$. Therefore, $g_M(X_1\cup X_2)=g_M(X_1)+g_M(X_2)$. Using (\ref{Neravenstvo1}) and (\ref{Neravenstvo2}), we obtain that  $g_M(X_1\cup X_2)>2(n-1)!$. Thus, $g_M(n)\geq g_M(X_1\cup X_2)>2(n-1)!$ and we obtain a contradiction with $g_M(n)\leq 2(n-1)!$.
\end{proof}

Let $x\in \mathbb{R}$, $x\neq 0$ and $p_1,p_2\in \{2,\ldots,n\}$, $p_1\neq p_2$.
We say that an $n\times n$ matrix $M=(m_{i,j})$ is the {\em $(x,p_1,p_2)$-matrix} if
$$
m_{i,j}=\begin{cases}
x,&\text{if $j=p_1$ and $i>1$;}\\
-x,&\text{if $j=p_2$ and $i>1$;}\\
0,&\text{otherwise.}
\end{cases}
$$

\begin{example}\label{4}
The matrix $$M=
\begin{pmatrix}
0 & 0 & 0 & 0\\
0 & 3 & 0 & -3\\
0 & 3 & 0 & -3\\
0 & 3 & 0 & -3
\end{pmatrix}$$
is the $(3,2,4)$-matrix.
\end{example}

We say that an $n\times n$ matrix $M$ belongs to the set $\mathcal{M}_1(n)$ if $M$ is the $(x,p_1,p_2)$-matrix for some $x\in \mathbb{R}$, $x\neq 0$ and $p_1,p_2\in \{2,\ldots,n\}$, $p_1\neq p_2$.

Let $x\in \mathbb{R}$, $x\neq 0$ and $q_1,q_2,\tau\in \{2,\ldots,n\}$, $q_1\neq q_2$.
We say that an $n\times n$ matrix $M$ is the {\em $(x,q_1,q_2,\tau)$-matrix} if
$$
m_{i,j}=\begin{cases}
x,&\text{if $i=\tau$ and $j=q_1$;}\\
-x,&\text{if $i=\tau$ and $j=q_2$;}\\
0,&\text{otherwise.}
\end{cases}
$$

\begin{example}\label{5}
The matrix $$M=
\begin{pmatrix}
0 & 0 & 0 & 0\\
0 & 0 & 0 & 0\\
0 & 5 & 0 & -5\\
0 & 0 & 0 & 0
\end{pmatrix}$$
is the $(5,2,4,3)$-matrix.
\end{example}

We say that an $n\times n$ matrix $M$ belongs to the set $\mathcal{M}_2(n)$ if $M$ is the $(x,q_1,q_2,\tau)$-matrix for some  $x\in \mathbb{R}$, $x\neq 0$ and $q_1,q_2,\tau\in \{2,\ldots,n\}$, $q_1\neq q_2$.

Now we prove the main theorem of this section.
\begin{theorem}\label{Theorem}
Let $M$ be a special $n\times n$ matrix, where $n\ge 8$ or $n=3$. Then $g_M(n)\ge 2(n-1)!$. Moreover, $g_{M}(n)=2(n-1)!$ if and only if $M\in \mathcal{M}_1(n)$ or $M\in \mathcal{M}_2(n)$.
\end{theorem}
\begin{proof}
One can verify that if $M\in \mathcal{M}_1(n)$ or $M\in \mathcal{M}_2(n)$, then $g_{M}(n)=2(n-1)!$.

Suppose that $g_M(n)\le 2(n-1)!$. Let us prove that $M\in \mathcal{M}_1(n)$ or $M\in \mathcal{M}_2(n)$.
Lemma \ref{RavnomernostMatrix} implies that $M$ is $\theta$-uniform for some $\theta\in{\{1,\ldots,n\}}$.
We consider two cases.

\textbf{Case 1.} Suppose $\theta=1$. Then $R_2=R_3=\ldots=R_n$.
By Lemma \ref{stringtype}, $R_2$ is the $(x,r_1,r_2)$-row or the $(y,s)$-row. If $R_2$ is the $(x,r_1,r_2)$-row, then $M$ is the $(x,r_1,r_2)$-matrix.
So, in this subcase $M\in \mathcal{M}_1(n)$. If $R_2$ is the $(y,s)$-row, then $g_M(n)=(n-1)\cdot(n-1)!$ and we obtain a contradiction with $g_M(n)\le 2(n-1)!$.

\textbf{Case 2.} Suppose $\theta>1$.
By Lemma \ref{stringtype}, $R_{\theta}$ is the $(x,r_1,r_2)$-row or the $(y,s)$-row. If $R_{\theta}$ is the $(x,r_1,r_2)$-row, then $M$ is the $(x,r_1,r_2,\theta)$-matrix.
So, in this subcase $M\in \mathcal{M}_2(n)$. If $R_\theta$ is the $(y,s)$-row, then $g_M(n)=(n-1)\cdot(n-1)!$ and we obtain a contradiction with $g_M(n)\le 2(n-1)!$.

\end{proof}

\section{Main Theorem}\label{SectionMainTheorem}
In this section we prove the main theorem of this paper.
\begin{theorem}\label{MainTheorem}
Let $f$ be an $(n-2)$-eigenfunction of $S_n$, where $n\ge 8$ or $n=3$. Then $|Supp(f)|\ge 2(n-1)!$. Moreover,  $|Supp(f)|=2(n-1)!$ if and only if $f=c\cdot \tilde{f}$, where $c$ is a real non-zero constant and $\tilde{f}\in \mathcal{F}$.
\end{theorem}
\begin{proof}
Lemma \ref{Supp(f)=g_M(n)} implies that $|Supp(f)|=g_{M(f)}(n)$.
We note that $M(f)$ is special. Then by Theorem \ref{Theorem} we obtain that $g_{M(f)}(n)\ge 2(n-1)!$. Therefore $|Supp(f)|\ge 2(n-1)!$.
Moreover, $g_{M(f)}(n)=2(n-1)!$ if and only if $M(f)\in \mathcal{M}_1(n)$ or $M(f)\in \mathcal{M}_2(n)$.
We consider two cases.

Suppose that $M(f)\in \mathcal{M}_1(n)$. Then $M(f)$ is the $(x,p_1,p_2)$-matrix.
Using Lemma \ref{f(pi)}, we have
$$
f(\pi)=\begin{cases}
x,&\text{if $\pi(p_2)=1$;}\\
-x,&\text{if $\pi(p_1)=1$;}\\
0,&\text{otherwise.}
\end{cases}
$$
So, in this case $f=x\cdot f_{1}^{p_2,p_1}$.

Suppose that $M(f)\in \mathcal{M}_2(n)$. Then $M(f)$ is the $(x,q_1,q_2,\tau)$-matrix.
Using Lemma \ref{f(pi)}, we have
$$
f(\pi)=\begin{cases}
x,&\text{if $\pi(q_1)=\tau$;}\\
-x,&\text{if $\pi(q_2)=\tau$;}\\
0,&\text{otherwise.}
\end{cases}
$$
So, in this case $f=x\cdot f_{\tau}^{q_1,q_2}$.

\end{proof}

\section{Correspondence between the extremal $(n-2)$-eigenfunctions and completely regular codes}\label{SectionCompletelyRegularCodes}
\begin{lemma}\label{CompletelyRegularCode}
Let $\alpha\in \{2,\ldots,n\}$, $a\in \{1,\ldots,n\}$ and $n\ge 3$. Then the set $S_{a}^{\alpha}$ is a completely regular code of covering radius $2$ in $S_n$.
\end{lemma}
\begin{proof}
Firstly, we note that $(S_{a}^{\alpha})^{(1)}=S_{a}^{1}$,
$$(S_{a}^{\alpha})^{(2)}=\bigcup_{\beta\in \{1,\ldots,n\}\setminus\{1,\alpha\}}S_{a}^{\beta}$$ and
$(S_{a}^{\alpha})^{(3)}=\emptyset$.
On the other hand, the partition $((S_{a}^{\alpha})^{(0)},(S_{a}^{\alpha})^{(1)},(S_{a}^{\alpha})^{(2)})$ has the
quotient matrix
$$\begin{pmatrix}
n-2 & 1 & 0\\
1 & 0 & n-2\\
0 & 1 & n-2\\
\end{pmatrix}.$$
So, $S_{a}^{\alpha}$ is a completely regular code of covering radius $2$ in $S_n$.
\end{proof}
For a set $A\subseteq Sym_n$ we define the characteristic function $\chi_{A}$ of $A$ in $Sym_n$ as follows:
$$
\chi_{A}(\pi)=\begin{cases}
1,&\text{if $\pi\in A$;}\\
0,&\text{otherwise.}
\end{cases}
$$

Using Theorem \ref{MainTheorem} and the definition of $f_{u}^{v,w}$, we immediately obtain the following result.

\begin{lemma}\label{RaznostKodov}
Let $f$ be an $(n-2)$-eigenfunction of $S_n$ and $|Supp(f)|=2(n-1)!$, where $n\ge 8$ or $n=3$. Then $$f=c\cdot (\chi_{S_{u}^{v}}-\chi_{S_{u}^{w}}),$$ where $c$ is a real non-zero constant, $u\in\{1,\ldots,n\}$, $v,w\in\{2,\ldots,n\}$ and $v\neq w$.
\end{lemma}

Thus, Lemma \ref{CompletelyRegularCode} and Lemma \ref{RaznostKodov} imply that for $n\ge 8$ and $n=3$ an arbitrary $(n-2)$-eigenfunction of $S_n$
with the minimum cardinality of the support is the difference of the characteristic functions of two completely regular codes of covering radius $2$.
It is very interesting that there is an analogue of this fact for the Hamming graph $H(n,q)$ (see \cite{Valyuzhenich}, Theorem 3).
Namely, an arbitrary eigenfunction of $H(n,q)$ corresponding to
the second largest eigenvalue with the minimum cardinality of the support is  the difference of the characteristic functions of two completely regular codes of covering radius $1$.

\section{Concluding remarks}
The initial problem of finding $(n-2)$-eigenfunctions of $S_n$ with the minimum size of the support is formulated for arbitrary real-valued functions from corresponding eigenspace. Surprisingly, Theorem \ref{MainTheorem} implies that such functions take only three distinct values.
It is interesting that the same fact holds for the Doob graph, for the Hamming graph and for the Johnson graph (see \cite{Bespalov,ValyuzhenichVorobev,VMV}).
But, in general case it is not true.
For example, in the Petersen and Desargues graphs there are $(-2)$-eigenfunctions with the minimum sizes of the supports that take five  distinct values (see \cite{SotnikovaCubical}, Figure 3 and 9).

We note that the restrictions for $n$ in Theorem \ref{MainTheorem} ($n\ge 8$ or $n=3$) arise from the proofs of Lemmas \ref{Arifmetica}, \ref{(A,B)Property} and \ref{RavnomernostMatrix}.

\section{Acknowledgements}

The authors are grateful to Sergey Goryainov for useful discussions.

\end{document}